\documentclass[12pt,reqno]{amsart}
\usepackage{amssymb,amsthm}
\usepackage[utf8]{inputenc}

\advance\hoffset by -0.5truein
\let\<\langle
\let\>\rangle

\def\Pois{\mathrm{Pois}}
\def\Lie{\mathrm{Lie}}
\def\adm{\mathrm{adm}}
\def\Anti{\mathrm{Anti}}

\def\As {\mathop {\fam0 As }\nolimits}
\def\GD {\mathop {\fam0 GD }\nolimits}

\def\Com{\mathrm{Com}}
\def\Nov{\mathrm{Nov}}

\newtheorem{theorem}{Theorem}
\newtheorem{lemma}{Lemma}

\newtheorem{corollary}{Corollary}

\theoremstyle{definition}
\newtheorem{definition}{Definition}

\newtheorem{example}{Example}
\usepackage[OT2,T1]{fontenc}
\DeclareSymbolFont{cyrletters}{OT2}{wncyr}{m}{n}
\DeclareMathSymbol{\Sha}{\mathalpha}{cyrletters}{"58}

\usepackage{graphicx}
\newcommand{\bcong}{\mathbin{\rotatebox[origin=c]{90}{$\cong$}}}

\title[Free product of operads and free basis of Lie-admissible operad]{Free 
product of operads and free basis of Lie-admissible operad}
\author[B.K. Sartayev]{B.K. Sartayev}

\address{Sobolev Institute of Mathematics, Novosibirsk, Russia and Astana IT university, Nur-Sultan, Kazakstan}
\email{baurjai@gmail.com}

\begin{document}
\begin{abstract}
In this paper, we introduce the definition of a free product of operads 
following the definition of a free product of algebras. We present a method for 
finding the basis and dimension of the free product of operads. We prove that 
the Lie-admissible operad is isomorphic to the free product of Lie and 
commutative (but non-associative) operads.
\end{abstract}


\maketitle

\section{Introduction}
A wide range of algebraic systems in nonassociative algebra is presented by 
linear spaces with two operations satisfying certain axioms.
So are Poisson algebras, Novikov--Poisson algebras \cite{Xu}, Gelfand--Dorfman 
(bi)algebras \cite{KSO2019}, \cite{Xu2000}.
Each operad governing either of the corresponding varieties of algebras is a 
quotient of the operad
$\mathcal O$ generated by two different operations
without identities that contain both these products.
A natural construction that describes such an operad from a more general point 
of view
is the free product of operads.

Obviously, if operads $\mathcal O_i$, $i=1,2$, have
the graded spaces of generators $V_i$
and defining relations $\Sigma_i \subset \bigcup\limits_{n\geq 1} \mathcal 
F(V_i)(n)$
then $\mathcal O_1\ast \mathcal O_2$
($\ast$ symbol of free product of operads)
is the operad generated by $V_1\oplus V_2$ with defining
relations $\Sigma_1\cup \Sigma _2$.
For example, the operad governing the variety of $2$-$\As$ algebras 
\cite{Loday-Ronco}
is a free product of two operads $\As$ governing the class of associative 
algebras.
We may also consider the class of $2$-$\Com$-$\As$ algebras as an analogue of 
$2$-$\As$.
There is a natural connection between $\mathcal O_1\ast \mathcal O_2$ and 
series-parallel networks.

The notion of a Lie-admissible algebra was introduced by A.~Albert.
By the definition, an algebra $\mathcal{L}$ with a single product $(x,y)\mapsto 
xy$
is Lie-admissible if and only if
the commutator algebra $\mathcal{L}^{(-)}$ with the product $[x,y]=xy-yx$ is a 
Lie algebra.
Lie admissible algebra and operad was studied in \cite{Remm2002} and 
\cite{Goze-Remm2004}. 
It would be interesting to find other free algebras $\mathcal{A}$ as a Lie 
admissible that is isomorphic to a free product of $\mathcal{A}^{(-)}$ and 
$\mathcal{A}^{(+)}$.

The theory of Gr\"obner bases is  a useful tool for solving the word problem and 
finding normal forms
in commutative algebras in the most efficient algorithmic way.
For noncommutative and nonassociative algebras the Gr\"obner--Shirshov bases 
theory
addresses the same problems (see \cite{BokChen2014_BMS}).
The theory of Gr\"obner bases for operads was established in 
\cite{Bremner-Dotsenko} and \cite{Dotsenko-Khoroshkin}. To prove the isomorphism 
of operads, we use the theory developed by them.

\section{Free product of operads}

Let $S_n$ stand for the symmetric group of the set $\{1,\dots , n\}$.
A symmetric operad $\mathcal P$ is a symmetric collection of a
$S_n$-modules $\mathcal P(n)$, $n\geq 1$,
equipped with linear composition maps
\[
\gamma^m _{n_1,\dots,n_m} :\mathcal P(m)\otimes \mathcal P(n_1)\otimes\ldots 
\otimes \mathcal P(n_m)\rightarrow \mathcal P(n_1+\ldots +n_m)
\]
which satisfy the associativity condition.
The space $\mathcal {P}(1)$ contains an element
$\textrm{id}$ that acts as an identity relative to the compositions.
Finally, the compositions are equivariant with respect to the symmetric group 
actions.

An operad ideal in an operad $\mathcal P$ is a collection of $S_n$-invariant 
subspaces
$I(n)\subseteq \mathcal P(n)$, $n\geq 1$, such that
\[
\begin{gathered}
\gamma^m_{n_1,\dots,n_m} (I(m), \mathcal P(n_1), \ldots ,\mathcal P(n_m)) 
\subseteq I(n_1+\dots +n_m), \\
\gamma^m_{n_1,\dots,n_m} (\mathcal P(m), \mathcal P(n_1),\ldots, I(n_k),\ldots 
,\mathcal P(n_m)) \subseteq I(n_1+\dots +n_m).
\end{gathered}
\]

A morphism of operads $\psi: \mathcal O\to \mathcal P$
is a collection of $S_n$-linear maps
$\psi(n): \mathcal O(n)\to \mathcal P(n)$, $n\geq 1$,
preserving the compositions and identity.

\begin{definition}
Let $\mathcal O_1$ and $\mathcal O_2$ be two operads.
Then the {\em free product} of two operads $O_1$  and $O_2$
is an operad $O=O_1\ast O_2$ satisfying the following conditions.
\begin{enumerate}
    \item
     There exist morphisms $\pi_i: \mathcal O_i \to O$
     for $i=1,2$;

     \item For every operad $\mathcal P$ and for every pair of
     morphisms $\psi_i: \mathcal O_i\to \mathcal P$
     there exists a unique morphism
     $\psi : \mathcal O\to \mathcal P$
     such that $\psi_i = \psi \pi_i$ for $i=1,2$.
\end{enumerate}
\end{definition}

For every graded space
$V = \bigoplus_{n\ge 1} V(n)$ there exists a uniquely defined
free operad $\mathcal F(V)$ generated by~$V$.
An operad ideal of $\mathcal F(V)$ may be presented as a minimal one that 
contains a given series of elements from
$\bigcup\limits_{n\geq 1} \mathcal F(V)(n)$.
Therefore, an operad may be defined by generators and relations.

For example, the operad $\Com$-$\As$
defining the variety of commutative-associative algebras
is generated by $V_1=V_1(2)$, where $\dim V_1(2)=1$,
this is a symmetric $S_2$-module.
The free operad  $\mathcal F(V_1)$ is exactly the operad
of commutative algebras.
The defining relations of $\Com$-$\As$ consist of associative identity:
\[
\gamma^2_{2,1} (\mu,\mu,1) + \gamma^2_{1,2} (\mu, 1,\mu).
\]
It is well-known that $\dim \Com$-$\As(n) = 1$.

The operad $\Lie $
of the variety of Lie algebras
is generated by $V_2=V_2(2)$, where $\dim V_2(2)=1$,
this is a skew-symmetric $S_2$-module.
The free operad $\mathcal F(V_2)$ is exactly the operad
of anti-commutative algebras. The set of defining relations
of $\Lie $ consists of the Jacobi identity:
\[
\gamma^2_{2,1} (\mu,\mu,1) + \gamma^2_{2,1} (\mu,\mu,1)^{(123)} + \gamma^2_{2,1} 
(\mu,\mu,1)^{(132)}.
\]
It is well-known that $\dim \Lie(n) = (n-1)!$.

The graphical presentation of the free product of two operads
naturally comes from the theory of Gr\"obner bases for operads.
Let $X(n)$ and $Y(n)$, $n\ge 1$, be linear bases of $\mathcal O_1(n)$ and 
$\mathcal O_2(n)$, respectively. Then each of the operads $\mathcal O_1$ or 
$\mathcal O_2$ may be considered as a shuffle operad
(see \cite{Bremner-Dotsenko} for the definition) generated by
$X = \bigcup\limits_{n\ge 1}X(n)$ or $Y= \bigcup\limits_{n\ge 1}Y(n)$,
respectively.
The sets of defining relations $\Sigma_1$ and $\Sigma _2$ then consist of all 
possible compositions of elements in $X$ and $Y$, respectively.
(It is similar to the observation that in every reasonable class of algebras
the multiplication table of an algebra is a Gr\"obner--Shirshov basis.)
It remains to describe those elements of the free operad generated by the linear 
span of $X\cup Y$ reduced with respect to $\Sigma_1\cup\Sigma_2$.

\begin{example}\label{exmp:Free-ns}
For every $n\ge 1$, consider the set $B(n)$
of planar rooted trees with $n$ leaves enumerated by integers $1,\dots, n$,
with internal vertices labelled by $\circ $ and $\bullet $
satisfying the following conditions:
\begin{enumerate}
 \item Each internal vertex has at least two leaves;
 \item There are no edges connecting internal vertices with similar labels.
\end{enumerate}
Let us define a composition of such trees. In order to compose a tree $t\in 
B(n)$
with a sequence of trees $t_i\in B(m_i)$, $i=1,\dots, n$, we first perform the 
ordinary grafting with re-numeration of leaves and then suppress those edges (if 
any) that connect vertices with similar edges. The resulting tree
$\gamma^n_{m_1,\dots ,m_n}(t,t_1,\dots , t_n)$ belongs to $B(m_1+\dots +m_n)$. 
For example,

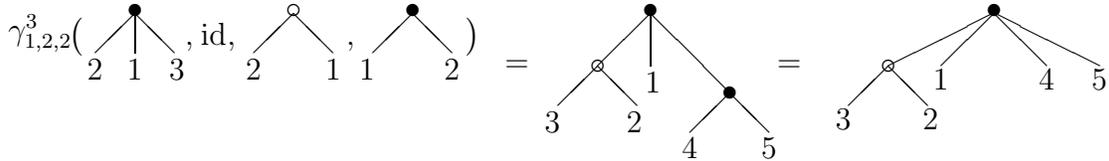
\begin{picture}(30,80)
\put(-13,60){$\gamma^3_{1,2,2}$}
\put(10,58){$\big($}
\put(54,58){$,$}
\put(32,68){$\bullet$}
\put(35,70){\line(-1,-1){15}}
\put(35,70){\line(1,-1){15}}
\put(35,70){\line(0,-1){15}}
\put(32,45){$1$}
\put(48,45){$3$}
\put(17,45){$2$}

\put(60,58){$\textrm{id}$}
\put(70,58){$,$}
\put(92,68){$\circ$}
\put(95,70){\line(-1,-1){15}}
\put(95,70){\line(1,-1){15}}
\put(107,45){$1$}
\put(77,45){$2$}

\put(115,58){$,$}
\put(137,68){$\bullet$}
\put(140,70){\line(-1,-1){15}}
\put(140,70){\line(1,-1){15}}
\put(152,45){$2$}
\put(120,45){$1$}
\put(160,58){$\big)$}

\put(175,48){$=$}
\put(227,68){$\bullet$}
\put(230,70){\line(-1,-1){20}}
\put(230,70){\line(0,-1){20}}
\put(230,70){\line(1,-1){30}}
\put(207,47){$\circ$}
\put(210,50){\line(-1,-1){15}}
\put(210,50){\line(1,-1){15}}
\put(257,37){$\bullet$}
\put(260,40){\line(-1,-1){15}}
\put(260,40){\line(1,-1){15}}

\put(190,25){$3$}
\put(221,25){$2$}
\put(228,40){$1$}
\put(242,15){$4$}
\put(272,15){$5$}
\put(278,48){$=$}

\put(357,68){$\bullet$}
\put(360,70){\line(-1,-1){20}}
\put(360,70){\line(-2,-1){40}}
\put(360,70){\line(1,-1){20}}
\put(360,70){\line(2,-1){40}}
\put(317,47){$\circ$}
\put(320,50){\line(-1,-1){15}}
\put(320,50){\line(1,-1){15}}

\put(300,26){$3$}
\put(333,26){$2$}
\put(337,41){$1$}
\put(377,41){$4$}
\put(397,41){$5$}
\end{picture}
\vspace*{-\baselineskip}

The family of spaces $B(n)$ together with the composition defined above
forms an operad~$\mathcal B$.

This is an easy exercise in the Gr\"obner bases technique for operads
to show $\mathcal B$ is isomorphic to $\As*\As$ (or two-associative algebra in 
\cite{Loday-Ronco}).
\end{example}

Suppose $\mathcal O_1$ and $\mathcal O_2$ are two binary symmetric  operads.
Let us construct an operad $\mathcal T = \mathcal T(\mathcal O_1,\mathcal O_2)$ 
as follows.
First, construct the family of linear spaces
$\mathcal T(n)$, $n\ge 1$, by induction. Set $\dim \mathcal T(1)=1$, so $\mathcal T(1)$ is 
the linear span of the identity~$\mathrm{id}=\mathrm{id}_T$.
For $n=2$, set $\mathcal T(2)=\mathcal T(2)^\bullet\oplus \mathcal T(2)^\circ$,
where $\mathcal T(2)^\bullet = \mathcal O_1(2)$,
$\mathcal T(2)^\circ = \mathcal O_2(2)$.

Suppose $n>2$. Consider all partitions $\lambda = (n_1,\dots, n_m)$,
$n_1+\dots +n_m=n$, where $n_1\ge \dots \ge n_m\ge 1$.
Assume each $\mathcal T(n_i)$ is already defined,
and for $n_i\ge 2$ it is presented in a form
\[
\mathcal T(n_i) = \mathcal T(n_i)^\bullet\oplus \mathcal T(n_i)^\circ.
\]
Then define
\[
\begin{gathered}
\mathcal T_\lambda (n)^\bullet
=\mathcal O_1(m)\otimes \mathcal T(n_1)^\circ \otimes
\dots \otimes \mathcal T(n_m)^\circ , \\
\mathcal T_\lambda (n)^\circ
=\mathcal O_2(m)\otimes \mathcal T(n_1)^\bullet \otimes
\dots \otimes \mathcal T(n_m)^\bullet .
\end{gathered}
\]
Denote by $S(m,\lambda )$ the set of all $\sigma \in S_m$
such that $n_{\sigma(i)}=n_i$ for all $i=1,\dots, m$.
There is a natural embedding
\[
S(\lambda ):=S(m,\lambda)\times S_{n_1}\times \dots \times S_{n_m} \to S_n
\]
given by the composition of permutations. The subgroup $S(\lambda )$ acts on
$\mathcal T_\lambda (n)^\bullet $ and
$\mathcal T_\lambda (n)^\circ $ in the obvious way.
Finally,
\[
\begin{gathered}
\mathcal T (n)^\bullet =
\bigoplus\limits_{\lambda \in P(n)} \mathrm{Ind}_{S(\lambda )}^{S_n}\mathcal 
T_\lambda (n)^\bullet , \quad
\mathcal T (n)^\circ =
\bigoplus\limits_{\lambda \in P(n)} \mathrm{Ind}_{S(\lambda )}^{S_n}\mathcal 
T_\lambda (n)^\circ , \\
\mathcal T(n) = \mathcal T (n)^\bullet \oplus \mathcal T (n)^\circ .
\end{gathered}
\]

In order to define a composition rule on $\mathcal T$, note that the elements
of $\mathcal T (n)$ are in one-to-one correspondence with formal linear
combinations of planar rooted trees from Example~\ref{exmp:Free-ns}
whose internal vertices have additional labels:
$\bullet $-vertices (resp.,  $\circ$-vertices) with $m$ leaves are marked by 
elements from
$\mathcal O_1(m)$ (resp., $\mathcal O_2(m)$).
The composition of such trees is a straightforward generalization of that from
Example~\ref{exmp:Free-ns} assuming that the labels of gluing vertices are 
composed accordingly. For example,


\begin{picture}(30,80)
\put(17,60){$\gamma^3_{1,2,2}$}
\put(40,58){$\big($}
\put(88,54){$\mu$}
\put(103,58){$,$}
\put(67,68){$\bullet$}
\put(70,70){\line(-1,-1){15}}
\put(70,70){\line(1,-1){15}}
\put(52,45){$1$}
\put(82,51){$\circ$}
\put(85,55){\line(-1,-1){15}}
\put(85,55){\line(1,-1){15}}
\put(97,30){$3$}
\put(67,30){$2$}

\put(110,58){$\textrm{id}$}
\put(120,58){$,$}
\put(152,70){$\nu$}
\put(145,68){$\circ$}
\put(148,70){\line(-1,-1){15}}
\put(148,70){\line(1,-1){15}}
\put(160,45){$1$}
\put(130,45){$2$}

\put(170,58){$,$}
\put(192,68){$\bullet$}
\put(195,70){\line(-1,-1){15}}
\put(195,70){\line(1,-1){15}}
\put(207,45){$2$}
\put(175,45){$1$}
\put(215,58){$\big)$}

\put(230,48){$=$}
\put(267,68){$\bullet$}
\put(270,70){\line(-1,-1){20}}
\put(270,70){\line(1,-1){20}}
\put(287,47){$\circ$}
\put(290,50){\line(-1,-1){15}}
\put(290,50){\line(0,-1){15}}
\put(290,50){\line(1,-1){15}}
\put(302,33){$\bullet$}
\put(305,35){\line(-1,-1){15}}
\put(305,35){\line(1,-1){15}}

\put(295,50){$\gamma_{2,1}^2(\mu,\nu,\textrm{id})$}
\put(245,40){$1$}
\put(272,27){$3$}
\put(287,27){$2$}
\put(287,10){$4$}
\put(317,10){$5$}
\end{picture}
\vspace*{-\baselineskip}

\begin{theorem}\label{th:FreeProdOperad}
The operad $\mathcal T=\mathcal T(\mathcal O_1,\mathcal O_2)$ constructed above 
is isomorphic
to $\mathcal O_1*\mathcal O_2$.
\end{theorem}

\begin{proof}
The trees representing basic elements of $\mathcal T(n)$
are exactly the reduced trees of the free operad generated by
$$
\bigoplus\limits_{n\ge 1} (O_1(n) \oplus  O_2(n))
$$
relative to the defining relations representing all possible compositions
on $\mathcal O_1$ and on $\mathcal O_2$. Neither of the compositions contains 
vertices
from both $\mathcal O_1$ and $\mathcal O_2$.
\end{proof}

Let us calculate $\dim \mathcal T(n)$ according to the construction from
 Theorem~\ref{th:FreeProdOperad}.
Denote
$\dim \mathcal O_1(m)=x_m$,
$\dim \mathcal O_2(k)=y_k$,
$\dim \mathcal{T}^\bullet (n) = d_n^\bullet$,
$\dim \mathcal{T}^\circ (n) = d_n^\circ $.
Then $\dim \mathcal T(n) = d_n^\bullet + d_n^\circ$ for
$n\ge 2$,
and
\[
\begin{gathered}
d_n^\bullet = \sum\limits_{\lambda \in P(n)}
 \dfrac{n!}{n_1!\cdots n_m! |S(m,\lambda )|}
  x_m d_{n_1}^\circ \dots d_{n_m}^\circ , \\
d_n^\circ = \sum\limits_{\lambda \in P(n)}
 \dfrac{n!}{n_1!\cdots n_m! |S(m,\lambda )|}
  y_m d_{n_1}^\bullet \dots d_{n_m}^\bullet. \\
\end{gathered}
\]
Assuming $d_2^\bullet  = x_2$, $d_2^\circ=y_2$,
we may calculate all dimensions in terms of $x_i$, $y_i$ by induction.

\begin{corollary}\label{Macmahon}
The dimensions $d^\bullet_n$ for $n=3,4,5$
are given by
\begin{align*}
 d^\bullet_3 ={} &x_3+3x_2y_2, \\
 d^\bullet_4={} & x_4+6x_3y_2+3x_2y_2^2+4x_2y_3+12x_2^2y_2,\\
d_5^\bullet = {}&
x_5 + 10x_4y_2 + 5x_2y_4 +15x_3y_2^2 + 30x_2^2y_3 + 50x_2x_3y_2 +10 x_2y_2y_3 \\
 &{} +90x_2^2y_2^2 +15x_2^3y_2 +10x_3y_3.
\end{align*}
In order to get $d^\circ_n$ one needs to interchange
$x$ and~$y$.
\end{corollary}

\begin{proof}
Let us show how to derive $d^\bullet_5$ from $d_k^\circ$, $2<k<5$.
There are six partitions of $n= 5$:
\begin{gather*}
    (4,1);\;\;\;(3,2); \\
    (3,1,1);\;\;\;(2,2,1);\\
    (2,1,1,1);\;\;\;\;\;\;(1,1,1,1,1).
\end{gather*}

For $\lambda =(4,1)$, $S(2,\lambda )=\{e\}$, so
$|S(\lambda )| = 1\times 4!\times 1 = 24$.
The dimension of the corresponding subspace
is $5x_2d_4^\circ $.

For $\lambda =(3,2)$, $S(2,\lambda )=\{e\}$, so
$|S(\lambda )| = 1\times 3!\times 2! = 12$.
The dimension of the corresponding subspace
is $10x_2d_3^\circ d_2^\circ $.

For $\lambda =(3,1,1)$, $S(3,\lambda )\simeq S_2$,
so  $|S(\lambda )| = 2!\times 3!\times 1\times 1 = 12$.
The dimension of the corresponding subspace
is $10x_3d_3^\circ $.

For $\lambda =(2,2,1)$, $S(3,\lambda )\simeq S_2$,
so  $|S(\lambda )| = 2!\times 2!\times 2!\times 1 = 8$.
The dimension of the corresponding subspace
is $15x_3d_2^\circ d_2^\circ $.

For $\lambda =(2,1,1,1)$, $S(4,\lambda )\simeq S_3$,
so  $|S(\lambda )| = 3!\times 2!\times 1\times 1 \times 1= 12$.
The dimension of the corresponding subspace
is $10x_4d_2^\circ $.

For $\lambda =(1,1,1,1,1)$, $S(5,\lambda )\simeq S_5$, so  $|S(\lambda )| = 
5!\times 1 = 120$.
The dimension of the corresponding subspace
is $x_5$.

Hence,
\begin{multline*}
d_5^\bullet
    = 5x_2d_4^\circ
    + 10x_2d_3^\circ d_2^\circ
    + 10x_3d_3^\circ
    + 15x_3d_2^\circ d_2^\circ
    + 10x_4d_2^\circ
    + x_5 \\
=
 5x_2(y_4+6y_3x_2+3y_2x_2^2+4y_2x_3+12y_2^2x_2)
 + 10x_2(y_3+3x_2y_2)y_2 \\
 + 10x_3 (y_3+3x_2y_2)
 + 15x_3y_2^2
 + 10x_4y_2
 +x_5 = x_5 + 10x_4y_2\\
 + 5x_2y_4 +15x_3y_2^2 + 30x_2^2y_3 + 50x_2x_3y_2 +10 x_2y_2y_3 +90x_2^2y_2^2 
+15x_2^3y_2 +10x_3y_3.
\end{multline*}
\end{proof}
In particular,
\begin{equation}\label{eq:d5-dim}
\begin{aligned}
d_5 = {}& d_5^\bullet + d_5^\circ
=x_5+y_5 + 15(x_4y_2+x_2y_4) + 20x_3y_3 \\
&{}+60(x_2x_3y_2+x_2y_2y_3) + 45(x_3y_2^2+x_2^2y_3)
+180 x_2^2y_2^2 + 15 (x_2^3y_2+y_2^3x_2).
\end{aligned}
\end{equation}

Given the construction of the basis for free product of operads is a useful tool 
to find the Gr\"obner basis of algebras equipped with two binary operations. Let 
us show some examples.

\begin{example}
Let us apply Corollary \ref{Macmahon} to calculate the dimension of the 
two-associative algebra up to degree 5:
\begin{center}
\begin{tabular}{c|ccccc}
 $n$ & 1 & 2 & 3 & 4 & 5 \\
 \hline
 $\dim(\As\ast \As(n)) $ & 1 & 4 & 36 & 528 & 10800
\end{tabular}
\end{center}
The obtained result is the same as the dimension of two-associative operad 
given in \cite{Loday-Ronco}.  
\end{example}

\begin{example}\label{ComLie}
Suppose that $\mathcal O_1$ is the operad $\Com$-$\As$ and $\mathcal O_2$ is the 
operad $\Lie$ generated by operations $\nu=(\cdot\cdot\cdot)$ and 
$\mu=[\cdot,\cdot]$, respectively. By Corollary \ref{Macmahon}, $\dim(\Lie\ast 
\Com$-$\As(4))=67$. Note that the operad $\Pois$ (governing Poisson algebras) 
can be defined as a quotient
\begin{center}
$\Lie\ast \Com$-$\As/ (\gamma^2_{1,2} (\mu,1,\nu)-\gamma^2_{2,1} 
(\nu,\mu,1)-\gamma^2_{2,1} (\nu,\mu,1)^{(23)}).$
\end{center}

It is well-known that $\dim\Pois(n)=n!$. The Leibniz identity written in terms of Lie 
brackets and associative-commutative multiplication can be converted 
into shuffle polynomials as follows:
\begin{gather*}
    [a,bc]\to [a,b]c+[a,c]b\;\;\;\Leftrightarrow\;\;\;[a,bc]\to [a,b]c+[a,c]b,\\
    [b,ac]\to [b,a]c+[b,c]a\;\;\;\Leftrightarrow\;\;\;[ac,b]\to [a,b]c-a[b,c],\\
    [c,ab]\to [c,a]b+[c,b]a\;\;\;\Leftrightarrow\;\;\;[ab,c]\to [a,c]b+a[b,c].
\end{gather*}
It is easy to show that the given rewriting system forms Gr\"obner basis 
generated by Leibniz identity. For that, we have to check $3$ compositions:
\begin{enumerate}
    \item $[a,bc]$ and $[ac,b]$. Those are $[ac,bd]$ and $[ad,bc]$.
    \item $[a,bc]$ and $[ab,c]$. That is $[ab,cd]$.
\end{enumerate}
Let us check only one of them (the remaining two can be checked in the same way). 
On the one hand,\\
$[ac,bd]\rightarrow [ac,d]b+[ac,b]d\rightarrow 
([a,d]c)b+(a[c,d])b+([a,b]c)d-(a[b,c])d\rightarrow 
([a,d]b)c+(ab)[c,d]+([a,b]c)d-(a[b,c])d$;\\
on the other hand,\\
$[ac,bd]\rightarrow [a,bd]c-a[bd,c]\rightarrow 
([a,b]d)c+([a,d]b)c-a([b,c]d)+a(b[c,d])\rightarrow 
([a,b]c)d+([a,d]b)c-(a([b,c])d+(ab)[c,d]$.

In other words, it means that we should reduce basic trees of 
$\Lie\ast\Com$-$\As$ that contain compositions of the form
$\mathcal O_2(m)\otimes \mathcal T(n_1)^\circ \otimes
\dots \otimes \mathcal T(n_m)^\circ$, where at least one $n_i>1$. There are 
exactly $43$ terms like that. 
So $\dim\Pois(4)=67-43=24$ as expected.
\end{example}

\begin{example}\label{gd4}
Let $\mu$ and $\nu$ be the generators of  the operads $\Lie$  and $\Nov$, 
respectively,
where $\Nov $ corresponds to the class of Novikov algebras.
Let us apply Corollary \ref{Macmahon} to calculate the dimensions
of the free product of operads $\Lie$ and $\Nov$ up to degree 5:
\begin{center}
\begin{tabular}{c|ccccc}
 $n$ & 1 & 2 & 3 & 4 & 5 \\
 \hline
 $\dim(\Lie\ast \Nov(n)) $ & 1 & 3 & 20 & 216 & 3274
\end{tabular}
\end{center}
In a similar way to the previous example, the quotient
\begin{equation}\label{eq:gd1}
\Lie\ast \Nov/([a,b\circ c]-[c,b\circ a]+[b,a]\circ c-[b,c]\circ a-b\circ [a,c])
\end{equation}
represents the operad $\GD$ governing the class of Gelfand--Dorfman algebras 
\cite{KSO2019}.

Let us find the dimensions of $\GD(n)$ up to $n=4$ as in the previous example 
omitting all the notions of a shuffle operad.
For degrees $n=1,2$ there is nothing to do.
Denote the identity from \eqref{eq:gd1} by $gd(1)$, and start with the set of 
identities $S:=\{gd(1)\}$.
If
$$
b\circ [a,c]\to [a,b\circ c]-[c,b\circ a]+[b,a]\circ c-[b,c]\circ a,
$$
then, for degree $3$, the rewriting system $S$ reduces elements of the form 
$a\circ [b,c]$.
There are exactly three different relations like that, so
$$
\dim(\GD(3))=20-3=17.
$$
Generally, by using $gd(1)$ we reduce all elements of the form $a\circ[b,c]$,
where $a$, $b$ and $c$ some monomials of the $\GD$-operad.

For degree $4$, consider the composition of
$(a\circ b)\circ c-(a\circ c)\circ b=0$ and $b=[u,v]$. 
Consider
$$(a\circ [u,v])\circ c-(a\circ c)\circ [u,v] \pmod S,$$
where $u<c<v$ and $\pmod S$ reduces all monomials by the rule $gd(1)$.
In this way, we obtain the following relation in $\GD$:
\begin{multline*}
gd(2) = [u,(a\circ c)\circ v]\to [v,(a\circ c)\circ u]+[u,a\circ c]\circ v-[v,a\circ 
c]\circ u \\
+ [u,a\circ v]\circ c-[v,a\circ u]\circ c+([a,u]\circ v)\circ c-([a,v]\circ 
u)\circ c.
\end{multline*}
Let us add $gd(2)$ to $S$, so now $S=\{gd(1), gd(2)\}$.

Computing the composition of
$(a\circ b)\circ c-(b\circ a)\circ c-a\circ(b\circ c)+b\circ(a\circ c)=0$ and 
$b=[u,v]$ modulo $S$
we get
\begin{multline*}
gd(3):=(a\circ[u,v])\circ c-([u,v]\circ a)\circ c-a\circ([u,v]\circ 
c)+[u,v]\circ(a\circ c)=[u,a\circ v]\circ c- \\
[v,a\circ u]\circ c+([a,u]\circ v)\circ c-([a,v]\circ u)\circ c-([u,v]\circ 
a)\circ c-a\circ([u,v]\circ c)+[u,v]\circ(a\circ c).
\end{multline*}
Add $gd(3)$ to the set $S$.
The rewriting rule corresponding to $gd(3)$ is
\begin{multline*}
a\circ([u,v]\circ c)\to [u,a\circ v]\circ c-[v,a\circ u]\circ c+([a,u]\circ 
v)\circ c- \\
([a,v]\circ u)\circ c-([u,v]\circ a)\circ c+[u,v]\circ(a\circ c).
\end{multline*}

Reducing elements of the form $x_1\circ[x_2,x_3]$ by $gd(1)$ and $[u,(a\circ 
c)\circ v]$ by $gd(2)$ (for $u<c<v$), and $a\circ([u,v]\circ c)$ by $gd(3)$, we 
get
$$
\dim(\GD(4))=\Lie\ast \Nov(4)/(gd(1),gd(2),gd(3))=
$$
$$
216-56(by\;gd(1))-8(by\;gd(2))-12(by\;gd(3))=140.
$$
The result agrees with the computer computation in \cite{KS2020} given dimension 
of $\GD$-operad up to degree 5, which coincides with the obtained result.
\end{example}

Let $2$-$\Com$-$\As\<X\>$ be the free algebra equipped with two binary 
operations generated
by the set $X=\<x_1,x_2,\ldots\>$,
where both multiplications are associative and commutative.

\begin{theorem}\label{th:MacMahon}
The dimension of the degree $n$ component of $2$-$\Com$-$\As\<x_1\>$ is equal to
the number of series-parallel networks
(or MacMahon numbers \cite{Macmahon numbers})
 with $n$ unlabeled edges.
\end{theorem}

\begin{proof}
By Theorem \ref{th:FreeProdOperad}, the operad $\Com$-$\As\ast\Com$-$\As$ is 
isomorphic to the multilinear part of algebra $2$-$\Com\<X\>$.
 To find the desired dimensions we should consider the basis trees of 
$\Com$-$\As\ast\Com$-$\As$ with unlabeled leaves. There is a natural bijection 
between these basis trees and series-parallel networks. To construct a bijection 
between them, we identify the first associative-commutative multiplication 
$\bullet$ with
 the parallel network  connection, and second associative-commutative 
multiplication $\circ$---with the series connection.
\end{proof}

Let us state the first seven terms of MacMahon numbers \cite{Macmahon numbers}:
\begin{center}
\begin{tabular}{c|cccccccc}
 $n$ unlabeled edges & 1 & 2 & 3 & 4 & 5 & 6 & 7 & $\ldots$ \\
 \hline
series-parallel networks  & 1 & 2 & 4 & 10 & 24 & 66 & 180 & $\ldots$
\end{tabular}
\end{center}

\begin{example}
If $n=2, 3$, then

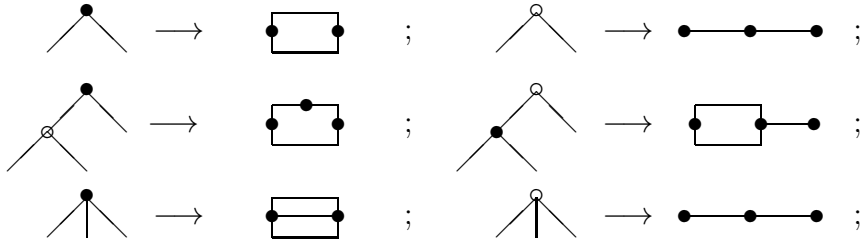
\begin{picture}(30,80)
\put(57,68){$\bullet$}
\put(60,70){\line(-1,-1){15}}
\put(60,70){\line(1,-1){15}}
\put(85,60){$\longrightarrow$}
\put(130,70){\line(1,0){25}}
\put(130,70){\line(0,-1){15}}
\put(130,55){\line(1,0){25}}
\put(155,70){\line(0,-1){15}}
\put(152,60){$\bullet$}
\put(127,60){$\bullet$}
\put(180,60){$;$}

\put(227,68){$\circ$}
\put(230,70){\line(-1,-1){15}}
\put(230,70){\line(1,-1){15}}
\put(255,60){$\longrightarrow$}
\put(285,63){\line(1,0){50}}
\put(283,60){$\bullet$}
\put(308,60){$\bullet$}
\put(333,60){$\bullet$}
\put(350,60){$;$}

\put(57,38){$\bullet$}
\put(60,40){\line(-1,-1){15}}
\put(60,40){\line(1,-1){15}}
\put(42,22){$\circ$}
\put(45,25){\line(-1,-1){15}}
\put(45,25){\line(1,-1){15}}
\put(83,25){$\longrightarrow$}
\put(130,35){\line(1,0){25}}
\put(130,35){\line(0,-1){15}}
\put(130,20){\line(1,0){25}}
\put(155,35){\line(0,-1){15}}
\put(152,25){$\bullet$}
\put(140,32){$\bullet$}
\put(127,25){$\bullet$}
\put(180,25){$;$}

\put(227,38){$\circ$}
\put(230,40){\line(-1,-1){15}}
\put(230,40){\line(1,-1){15}}
\put(212,22){$\bullet$}
\put(215,25){\line(-1,-1){15}}
\put(215,25){\line(1,-1){15}}
\put(255,25){$\longrightarrow$}
\put(290,35){\line(1,0){25}}
\put(290,35){\line(0,-1){15}}
\put(290,20){\line(1,0){25}}
\put(315,35){\line(0,-1){15}}
\put(315,28){\line(1,0){20}}
\put(312,25){$\bullet$}
\put(287,25){$\bullet$}
\put(332,25){$\bullet$}
\put(350,25){$;$}

\put(57,-2){$\bullet$}
\put(60,0){\line(-1,-1){15}}
\put(60,0){\line(1,-1){15}}
\put(60,0){\line(0,-1){15}}
\put(85,-10){$\longrightarrow$}
\put(130,0){\line(1,0){25}}
\put(130,-7){\line(1,0){25}}
\put(130,0){\line(0,-1){15}}
\put(130,-15){\line(1,0){25}}
\put(155,0){\line(0,-1){15}}
\put(152,-10){$\bullet$}
\put(127,-10){$\bullet$}
\put(180,-10){$;$}

\put(227,-2){$\circ$}
\put(230,00){\line(-1,-1){15}}
\put(230,00){\line(1,-1){15}}
\put(230,00){\line(0,-1){15}}
\put(255,-10){$\longrightarrow$}
\put(285,-7){\line(1,0){50}}
\put(283,-10){$\bullet$}
\put(308,-10){$\bullet$}
\put(333,-10){$\bullet$}
\put(350,-10){$;$}
\end{picture}
\end{example}
$ $

%
For every two operads $\mathcal{O}_1$ and $\mathcal{O}_2$ there is a natural mapping 
from monomials in $(\mathcal{O}_1\ast\mathcal{O}_2)(n)$ to the set of series-parallel networks with $n$ edges
as given above.
For every particular series-parallel network $B_i$, its pre-image span a subspace $\mathcal{N}_i$
in $(\mathcal{O}_1\ast \mathcal{O}_2)(n)$
The sum of all such subspaces form the entire space $(\mathcal{O}_1\ast \mathcal{O}_2)(n)$.
Then:
$$
(\mathcal{O}_1\ast\mathcal{O}_2)(n)=\mathcal{N}_1\oplus\mathcal{N}
_2\oplus\ldots\oplus\mathcal{N}_{k_n},
$$
where $k_n$ are the MacMahon numbers.

\section{Basis of Lie-admissible operad}

\begin{definition}
A Lie-admissible algebra over a field $\mathbb{K}$ of characteristic not equal to 2 is a 
vector space equipped with a binary operation satisfying the following 
identity:
\begin{multline}\label{Lie-admissible}
(ab)c=(ba)c+c(ab)-c(ba)-(bc)a+(cb)a+a(bc)-a(cb) \\
-(ca)b+(ac)b+b(ca)-b(ac).  
\end{multline}
\end{definition}

%
If $\mathcal{L}$ is a Lie-admissible algebra generated by a set 
$X=\{x_1,x_2,\ldots\}$ then $\mathcal{L}^{(-)}$ is a subalgebra of $\mathcal{L}$ 
generated by $X$ relative to the multiplication $[\cdot,\cdot]$, such that
$$[x_i,x_j]=x_ix_j-x_jx_i.$$
It is well-known that $\mathcal{L}^{(-)}$ is a Lie algebra and 
$\mathcal{L}^{(+)}$ is a commutative(non-associative) algebra, where 
$\mathcal{L}^{(+)}\<X\>$ is a subalgebra of $\mathcal{L}$ defined by 
multiplication $\{\cdot,\cdot\}$, such that:
$$
\{x_i,x_j\}=x_ix_j+x_jx_i.
$$

The operad $\Lie$-$\adm$ of the variety of Lie-admissible algebras
is generated by $W_2=W_2(2)$, where $\dim W_2(2)=2$.
The free operad $\mathcal F(W_2)$ is exactly the operad of magma algebras.
The set of defining relations
of $\Lie$-$\adm$ consists of the identity \eqref{Lie-admissible}:
\begin{multline}\label{operadLie-Admissible}
\gamma^2_{2,1} (\rho,\rho,1) - \gamma^2_{2,1} (\rho,\rho,1)^{(12)} - 
\gamma^2_{1,2} (\rho,1,\rho)^{(132)} + \gamma^2_{1,2} (\rho,1,\rho)^{(13)} \\
+\gamma^2_{2,1} (\rho,\rho,1)^{(123)} - \gamma^2_{2,1} (\rho,\rho,1)^{(13)} - 
\gamma^2_{1,2} (\rho,1,\rho) + \gamma^2_{1,2} (\rho,1,\rho)^{(23)}\\
+\gamma^2_{2,1} (\rho,\rho,1)^{(132)} - \gamma^2_{2,1} (\rho,\rho,1)^{(23)} - 
\gamma^2_{1,2} (\rho,1,\rho)^{(123)} + \gamma^2_{1,2} (\rho,1,\rho)^{(12)}
\end{multline}

\begin{definition}\label{shuffle operad}
A shuffle operad is a monoid in the category of nonsymmetric collections of 
vector spaces
 with respect to the shuffle composition product defined as follows:
$$
\mathcal{V}\circ_\Sha\mathcal{W}(n)=\bigoplus_{r\geq 
1}\mathcal{V}(r)\otimes\bigoplus_{\pi}\mathcal{W}(|I^{(1)}|)\otimes 
\mathcal{W}(|I^{(2)}|)\otimes\ldots\otimes\mathcal{W}(|I^{(r)}|),
$$
where $\mathcal{W}$ and $\mathcal{V}$ are nonsymmetric collections, $\pi$
ranges in all set partitions $\{1,\ldots,n\}$ $=\bigsqcup_{j=1}^r I^{(j)}$
for which all parts $I^{j}$ are nonempty and 
$\mathrm{min}(I_{1})<\ldots<\mathrm{min}(I_{r})$.
\end{definition}

\begin{lemma}\label{com*anti-com}
The operads $\Com$ and $\Com\ast\Anti$-$\Com$ may be presented as free shuffle 
operads
with operation alphabet $\mathcal{X}(2)=\{x\}$
and $\mathcal{X}(2)=\{x,y\}$,
respectively.
\end{lemma}

\begin{proof}
Firstly, let us prove this statement for operad $\Com$ by induction of a 
monomial length. Suppose that a generator of operad $\Com$ is $\mu$. For $n=2$, 
there are two monomials: $\mu(1,2)\leftrightarrow x(1 2)$, 
$\mu(2,1)=\mu(1,2)\leftrightarrow x(1 2)$.
Every monomial of length $n$ can be presented as: $\mu(k,\mu(A_1))$, 
$\mu(\mu(A_2),l)$ $\mu(\mu(A_3),\mu(A_4))$, where $1\leq k,l\leq n$ and $A_i$ 
are monomials of $\Com$ that length less that $n$. To prove this statement, we 
have to consider the following cases:
\begin{enumerate}
    \item For $\mu(k,\mu(A_1))$, if $k=1$, then $\mu(1,\mu(A_1))\in 
\mathcal{T}_{\Sha}(\mathcal{X})$ by the inductive hypothesis.
    \item For $\mu(k,\mu(A_1))$, if $k\neq 1$, then 
$\mu(k,\mu(A_1))=\mu(\mu(A_1),k)\in \mathcal{T}_{\Sha}(\mathcal{X})$ by the 
inductive hypothesis.
    \item For $\mu(\mu(A_2),l)$, if $l=1$, then 
$\mu(\mu(A_2),1)=\mu(1,\mu(A_2))\in \mathcal{T}_{\Sha}(\mathcal{X})$ by the 
inductive hypothesis.
    \item For $\mu(\mu(A_2),l)$, if $l\neq 1$, then $\mu(\mu(A_2),l)\in 
\mathcal{T}_{\Sha}(\mathcal{X})$ by the inductive hypothesis.
    \item For $\mu(\mu(A_3),\mu(A_4))$, if monomial $A_3$ contains index $1$, 
then $\mu(\mu(A_3),\mu(A_4))$ $\in \mathcal{T}_{\Sha}(\mathcal{X})$ by the 
inductive hypothesis.
    \item For $\mu(\mu(A_3),\mu(A_4))$, if monomial $A_4$ contains index $1$, 
then $\mu(\mu(A_3),\mu(A_4))$ $=\mu(\mu(A_4),\mu(A_3))\in 
\mathcal{T}_{\Sha}(\mathcal{X})$ by the inductive hypothesis.
\end{enumerate}
In the same way, this statement can be proved for the case 
$\Com\ast\Anti$-$\Com$.
\end{proof}

\begin{example}\label{shufflelie}
Let us consider the operad $\Lie$ with operation $\mu$.
 Using a forgetful functor $f$, we can convert this operad to the shuffle operad 
$\mathcal{T}_{\Sha}(\mathcal{X})$ as follows  \cite{Bremner-Dotsenko}:
$$
\mu(1\; 2)\longleftrightarrow x(1\; 2),\;\;\;\mu(2\; 1)\longleftrightarrow y(1\; 
2),
$$
for the operation alphabet $\mathcal{X}(2)=\{x,y\}$.
By anti-commutativity, $x(1\; 2)=-y(1\; 2)$.
For $\mathcal{T}_{\Sha}(\mathcal{X})$, the Jacobi identity can be written in the 
following form:
$$
x(x(1\; 2)\; 3)-x(1\; x(2\; 3))-x(x(1\; 3)\; 2).
$$
This element generates an ideal of relations defining the operad $\Lie^f$
as a quotient of $\mathcal{T}_{\Sha}(\mathcal{X})$.

To find the Gr\"obner basis of $\Lie^f$ we have to check only one composition of 
the Jacobi identity: $x(x(x(1\; 2)\; 3)\; 4)$,
where the leading monomial is $x(x(1\; 2)\; 3)$. On the one hand,\\
$x(x(x(1\; 2)\; 3)\; 4)\rightarrow x(x(x(1\; 2)\; 4)\; 3)+x(x(1\; 2)\; x(3\; 
4))\rightarrow x(x(x(1\; 4)\; 2)\; 3)+x(x(1\; x(2\; 4))\; 3)+x(x(1\; x(3\; 4))\; 
2)+x(1\; x(2\; x(3\; 4)))\rightarrow
x(x(x(1\; 4)\; 3)\; 2)+x(x(1\; x(3\; 4))\; 2)+x(x(1\; 3)\; x(2\; 4))+x(x(1\; 
4)\; x(2\; 3))+x(1\; x(x(2\; 4)\; 3))+x(1\; x(2\; x(3\; 4)))$,\\
on the other hand,\\
$x(x(x(1\; 2)\; 3)\; 4)\rightarrow x(x(x(1\; 3)\; 2)\; 4)+x(x(1\; x(2\; 3))\; 
4)\rightarrow x(x(x(1\; 3)\; 4)\; 2)+x(x(1\; 3)\; x(2\; 4))+x(x(1\; 4)\; x(2\; 
3))+x(1\; x(x(2\; 3)\; 4))\rightarrow x(x(x(1\; 4)\; 3)\; 2)+x(x(1\; x(3\; 4))\; 
2)+x(x(1\; 3)\; x(2\; 4))+x(x(1\; 4)\; x(2\; 3))+x(1\; x(x(2\; 4)\; 3))+x(1\; 
x(2\; x(3\; 4))).$

The final expressions coincide, so the single defining relation of the operad 
$\Lie^f$ forms a Gr\"obner basis.
\end{example}

Let us apply the forgetful functor $f$ to the operad $\Lie$-$\adm$ and find it's 
Gr\"obner basis. Choose the operation alphabet $\mathcal{X}$ with
 $\mathcal{X}(2)=\{x,y\}$ to convert the operad $\Lie$-$\adm$ to shuffle operad 
$\mathcal{T}_{\Sha}(\mathcal{X})$ as in Example \ref{shufflelie}:
$$x(1\;2)\longleftrightarrow \rho(1\;2),\quad y(1\;2)\longleftrightarrow 
\rho(2\;1).$$
Converting the identity (\ref{operadLie-Admissible}) of the operad $\Lie$-$\adm$
into an element of $\mathcal{T}_\Sha(\mathcal{X})$, we obtain the following:
\begin{multline}\label{shufflelieadm}
x(x(1\; 2)\; 3) - x(y(1\; 2)\; 3) - y(x(1\; 2)\; 3) + y(y(1\; 2)\; 3) + y(1\; 
x(2\; 3)) - y(1\; y(2\; 3)) \\
- x(1\; x(2\; 3)) + x(1\; y(2\; 3)) + x(y(1\; 3)\; 2) - x(x(1\; 3)\; 2) - 
y(y(1\; 3)\; 2) + y(x(1\; 3)\; 2)
\end{multline}
This element generates the ideal of relations defining the
 quotient of $\mathcal{T}_{\Sha}(\mathcal{X})$ isomorphic to the operad 
$\Lie$-$\adm^f$.

Now, we are ready to find a Gr\"obner basis generated by relation 
(\ref{operadLie-Admissible}).

\begin{lemma}\label{BG-lie-adm}
The defining relation (\ref{shufflelieadm}) of the operad $\Lie$-$\adm^f$ is a 
Gr\"obner basis.
\end{lemma}

\begin{proof}
The leading monomial of (\ref{shufflelieadm}) is $x(x(1\; 2)\; 3)$.
There is only one composition that we have to check: $x(x(x(1\; 2)\; 3)\; 4)$.
On the one hand,
\eqref{shufflelieadm} implies
\begin{eqnarray*}
x(x(x(1\; 2)\; 3)\; 4)\longrightarrow x(x(1\; 2)\; x(3\; 4)) - x(x(1\; 2)\; 
y(3\; 4)) + x(x(x(1\; 2)\; 4)\; 3)\\
- x(y(x(1\; 2)\; 4)\; 3) + x(y(x(1\; 2)\; 3)\; 4) - y(x(1\; 2)\; x(3\; 4)) + 
y(x(1\; 2)\; y(3\; 4))\\
-y(x(x(1\; 2)\; 4)\; 3) + y(x(x(1\; 2)\; 3)\; 4) +
 y(y(x(1\; 2)\; 4)\; 3) - y(y(x(1\; 2)\; 3)\; 4) \\
\longrightarrow x(1\; x(2\; x(3\; 4))) - x(1\; x(2\; y(3\; 4))) + x(1\; x(x(2\; 
4)\; 3)) -
 x(1\; x(y(2\; 4)\; 3))\\
  - x(1\; y(2\; x(3\; 4))) + x(1\; y(2\; y(3\; 4))) -
 x(1\; y(x(2\; 4)\; 3)) + x(1\; y(y(2\; 4)\; 3))\\
+ x(x(1\; 4)\; x(2\; 3)) -
 x(x(1\; 4)\; y(2\; 3)) + x(x(1\; 3)\; x(2\; 4)) - x(x(1\; 3)\; y(2\; 4)) \\
+ x(x(1\; x(3\; 4))\; 2) - x(x(1\; y(3\; 4))\; 2) + x(x(x(1\; 4)\; 3)\; 2) -
 x(x(y(1\; 4)\; 3)\; 2) \\+ x(x(y(1\; 2)\; 4)\; 3) - x(y(1\; 4)\; x(2\; 3)) +
 x(y(1\; 4)\; y(2\; 3)) + x(y(1\; 2)\; x(3\; 4))\\
 -x(y(1\; 2)\; y(3\; 4)) -
 x(y(1\; 3)\; x(2\; 4)) + x(y(1\; 3)\; y(2\; 4)) - x(y(1\; x(3\; 4))\; 2) \\
+ x(y(1\; y(3\; 4))\; 2) - x(y(x(1\; 4)\; 3)\; 2) + x(y(x(1\; 2)\; 3)\; 4) +  
x(y(y(1\; 4)\; 3)\; 2)\\
 - x(y(y(1\; 2)\; 4)\; 3) - y(1\; x(2\; x(3\; 4))) +
 y(1\; x(2\; y(3\; 4))) - y(1\; x(x(2\; 4)\; 3)) \\
 + y(1\; x(y(2\; 4)\; 3)) +
 y(1\; y(2\; x(3\; 4))) - y(1\; y(2\; y(3\; 4))) + y(1\; y(x(2\; 4)\; 3)) \\
- y(1\; y(y(2\; 4)\; 3)) - y(x(1\; 4)\; x(2\; 3)) + y(x(1\; 4)\; y(2\; 3)) -
 y(x(1\; 3)\; x(2\; 4)) \\
 + y(x(1\; 3)\; y(2\; 4)) + y(x(1\; x(2\; 3))\; 4) -
 y(x(1\; x(3\; 4))\; 2) - y(x(1\; y(2\; 3))\; 4) \\
 + y(x(1\; y(3\; 4))\; 2) -
 y(x(x(1\; 4)\; 3)\; 2) + y(x(x(1\; 3)\; 2)\; 4) + y(x(y(1\; 4)\; 3)\; 2)  \\
- y(x(y(1\; 2)\; 4)\; 3) + y(x(y(1\; 2)\; 3)\; 4) - y(x(y(1\; 3)\; 2)\; 4) +
 y(y(1\; 4)\; x(2\; 3)) \\- y(y(1\; 4)\; y(2\; 3)) - y(y(1\; 2)\; x(3\; 4)) +
 y(y(1\; 2)\; y(3\; 4)) + y(y(1\; 3)\; x(2\; 4))\\ - y(y(1\; 3)\; y(2\; 4)) -
 y(y(1\; x(2\; 3))\; 4) + y(y(1\; x(3\; 4))\; 2) + y(y(1\; y(2\; 3))\; 4) \\
- y(y(1\; y(3\; 4))\; 2) + y(y(x(1\; 4)\; 3)\; 2) - y(y(x(1\; 3)\; 2)\; 4) -
 y(y(y(1\; 4)\; 3)\; 2) \\ + y(y(y(1\; 2)\; 4)\; 3) - y(y(y(1\; 2)\; 3)\; 4) +
 y(y(y(1\; 3)\; 2)\; 4)=S.
\end{eqnarray*}
On the other hand,
\begin{eqnarray*}
x(x(x(1\; 2)\; 3)\; 4)\rightarrow x(x(1\; x(2\; 3))\; 4) - x(x(1\; y(2\; 3))\; 
4) + x(x(x(1\; 3)\; 2)\; 4)\\ + x(x(y(1\; 2)\; 3)\; 4) - x(x(y(1\; 3)\; 2)\; 4) 
- x(y(1\; x(2\; 3))\; 4) +
x(y(1\; y(2\; 3))\; 4)\\
+ x(y(x(1\; 2)\; 3)\; 4) - x(y(x(1\; 3)\; 2)\; 4) -
x(y(y(1\; 2)\; 3)\; 4) + x(y(y(1\; 3)\; 2)\; 4)\rightarrow S.
\end{eqnarray*}
\end{proof}

\begin{theorem}
The operad $\Lie$-$\adm$ is isomorphic to the free product $\Lie\ast\Com$.
\end{theorem}

\begin{proof}
%
%
If the operations of the anti-commutative ($\Anti$-$\Com$) and commutative \; 
($\Com$) operads are denoted by $x$ and $y$, respectively, then by Lemma 
\ref{com*anti-com},
$\Com\ast\Anti$-$\Com$ corresponds to  all monomials in the free shuffle operad 
with operation alphabet
 $\mathcal{X}(2)=\{x,y\}$.
 So
\begin{center}
    $\dim(\Com\ast\Anti$-$\Com)=\dim(\mathcal{T}_{\Sha}(\mathcal{X})).$
\end{center}
By Lemma \ref{BG-lie-adm}, we can reduce all trees in
$\mathcal{T}_{\Sha}(\mathcal{X})$
that contain a shuffle subtree of the form $x(x(1\; 2)\; 3)$.

By Example \ref{shufflelie}, the basis elements of the shuffle operad $\Lie^f$
with operation alphabet $\mathcal{X}(2)=\{x\}$
are those shuffle trees that contain no shuffle subtrees of the form
$x(x(1\; 2)\; 3)$.

Hence, the reduced form of a tree in $\Com\ast \Lie$ is the same as in 
$\Lie$-$\adm$,
and
\begin{center}
    $\dim(\Com\ast\Lie)=\dim(\Lie$-$\adm).$
\end{center}

Finally, by the definition of the free product
there is a morphism of operads $\Lie\ast\Com\to \Lie$-$\adm$
sending the generators of $\Lie$ and $\Com$ to $x(1\; 2)-y(1\;2)$ and 
$x(1\;2)+y(1\;2)$, respectively.
The morphism is surjective since every monomial in the free Lie-admissible algebra can be written as a 
 sum of monomials obtained from an associative word by applying $[\cdot,\cdot]$ and $\{\cdot,\cdot\}$
via $$ab=\frac{\{a,b\}+[a,b]}{2}$$
and $$ba=\frac{\{b,a\}+[b,a]}{2}.$$
The morphism is injective since the dimension of $\Com\ast\Lie$ is the same as the dimension of $\Lie$-$\adm$. 

Graphically it can be presented as follows:

\begin{picture}(30,80)
\put(190,60){\line(1,0){45}}
\put(230,57){$\rightarrow$}
\put(135,44){$\bcong$}
\put(190,35){\line(1,0){45}}
\put(230,32){$\rightarrow$}
\put(260,44){$\bcong$}
\put(98,57){$\Com\ast\Anti$-$\Com$}
\put(118,31){Shuffle($x,y$)}
\put(248,57){$\Com\ast\Lie$}
\put(248,31){$\Lie$-$\adm$}
\put(210,37){$g$}
\put(210,62){$f$}
\end{picture}
\vspace*{-\baselineskip}

\noindent where kernel of functions $f$ and $g$ are ideals generated by Jacoby 
identity and ($\ref{shufflelieadm}$), respectively.
\end{proof}

Using the calculation method of dimension $\Lie\ast\Com$, there can be obtained 
the dimension of Lie-admissible operad:

\begin{center}
\begin{tabular}{c|cccccccc}
 $n$ & 1 & 2 & 3 & 4 & 5 & 6 & 7 & \ldots \\
 \hline
 $\dim(\Lie\ast \Com(n)) $ & 1 & 2 & 11 & 101 & 1299 & 21484 & 434314 & \ldots
\end{tabular}
\end{center}

$ $

\textbf{Acknowledgments.} The author also expresses his gratitude to P. S. Kolesnikov, whose remarks helped to improve this paper.

\end{document}